\documentclass[10pt,reqno]{amsart}
\usepackage{amsmath,amsopn,amssymb,amsthm}
\usepackage[cp1251]{inputenc}
\usepackage[english]{babel}
\usepackage[pagebackref,breaklinks=true,colorlinks=true,linkcolor=blue,citecolor=blue,urlcolor=blue]{hyperref}
\usepackage{graphicx}%
\usepackage{color}

\newtheorem{theorem}{Theorem}[section]

\newtheorem{lemma}[theorem]{Lemma}

\newtheorem{remark}[theorem]{Remark}

\renewenvironment{proof}[1][Proof]{\noindent\textbf{#1.} }{\ \rule{0.5em}{0.5em}}

\newcommand{\G}{{\mathrm{G}}}
\begin{document}

\title[Left invariant complex Finsler metrics on a complex Lie group]{Left invariant complex Finsler metrics on a complex Lie group}
\author{Xiyun Xu}
\address[Xiyun Xu] {School of Mathematical Sciences,
Capital Normal University,
Beijing 100048,
P.R. China}
\email{2210502058@cnu.edu.cn}
\author{Ming Xu}
\address[Ming Xu] {Corresponding author, School of Mathematical Sciences,
Capital Normal University,
Beijing 100048,
P.R. China}
\email{mgmgmgxu@163.com}

\begin{abstract}
In this paper, we consider a left invariant complex Finsler metric $F$ on a complex Lie group. Using the technique of invariant frames, we prove the following properties for $(G,F)$. First, the metric $F$ must be a complex Berwald metric. Second, its complex spray $\chi=w^i\delta_{z^i}$ on $T^{1,0}G\backslash0$ can be extended to a holomorphic tangent field on $T^{1,0}G$. If we view $\chi$ as a real tangent field on $TG$, it coincides with the canonical bi-invariant spray structure on $G$. Third,
we prove that the strongly K\"{a}hler, K\"{a}hler, and weakly K\"{a}hler properties for $F$ are equivalent. More over, $F$ is K\"{a}hler if and only if $G$ has an Abelian Lie algebra. Finally, we prove that the holomorphic sectional curvature vanishes.

\vbox{}
\end{abstract}

\maketitle

\section{Introduction and main results}

The research on complex Finsler geometry has a history of almost sixty years
\cite{Fu1989}\cite{Ru1967}, in which the discovery of the Kobayashi metric
marks a milestone \cite{Le1981}\cite{Ko1975}. Since the 1990's,
complex Finsler geometry has been systematically explored by T. Aikou, M. Abata et al \cite{Ab1996}\cite{Ai1991}. Its connections to complex geometry and algebraic geometry have been revealed \cite{Ai2004}\cite{FLW2019}.  The K\"{a}hler properties, complex Berwald property, the holomorphic sectional curvature, etc.,
have been extensively studied \cite{CS2009}\cite{LQZ2024}\cite{LLQZ2025}\cite{Wa2019}\cite{XZ2015}\cite{Ya2008}.

The motivation of this paper is to generalize the homogeneous Finsler geometry \cite{De2012} to the complex context. Unlike the method in \cite{CGZ2024}, using special local coordinate in the calculation, we pursue a more algebraic approach which is inherited from our previous works  \cite{Xu2022-1}\cite{Xu2022}\cite{Xu2023}.
We start with a special case, i.e., we consider a left invariant complex Finsler metric $F$ on a complex Lie group $G$. In this paper, we automatically assume a complex Finsler metric is pseudo-convex.

The thought in \cite{Xu2022-1} help us construct the invariant frames on $T^{1,0}G$, which is as follows. Let $\{e_i,\forall 1\leq i\leq n\}$ be a basis for the complex Lie algebra $\mathfrak{g}=\mathrm{Lie}(G)$. Then we have left and right invariant holomorphic tangent fields, $U_i$ and $V_i$ respectively, such that the values of $U_i$ and $V_i$ at the identity $e\in G$ coincide with $e_i\in\mathfrak{g}=T^{1,0}G$. Correspondingly, we have holomorphic functions $u^i$ and $v^i$ on $T^{1,0}G$ such that each vector in $T^{1,0}_gG$ can be simultaneously presented as $u^iU_i(g)$ and $v^iV_i(g)$. In each $T^{1,0}_gG$, we have $\partial_{u^i}$ and $\partial_{v^i}$, which can be globalized to holomorphic tangent fields
on $T^{1,0}G$. Meanwhile, we can lift $U_i$ and $V_i$ to $\widetilde{U}_i$ and $\widetilde{V}_i$ respectively, which are also holomorphic tangent fields on $T^{1,0}G$. Then we get the left invariant frame $\{\widetilde{U}_i,\partial_{u^i},\forall i\}$ and the right invariant frame
$\{\widetilde{V}_i,\partial_{v^i},\forall i\}$.

Though the story seems similar here, we must be aware that there exist subtle differences between real and complex situations. Fortunately
the holomorphic property helps us overcome all obstacles. We prove that
the previously described complex invariant frames share similar properties as their real analogs in \cite{Xu2022-1} (see the lemmas in Section 3). As they are rarely seen in literature, we provide detailed proofs for all the preparation lemmas.

Using above invariant frames, we prove the following main theorems.
First, we prove

\begin{theorem}\label{main-thm-1}
A left invariant complex Finsler metric on a complex Lie group satisfies the complex Berwald property.
\end{theorem}

Theorem \ref{main-thm-1} reveals an intrinsic difference between the real Berwald property and the complex Berwald property. As I know, homogeneous real Finsler metrics which are Berwald but not Riemannian are relatively rare.

Second, we consider the complex spray $\chi=w^i\delta_{z^i}$ for the left invariant complex Finsler metric $F$ on a complex Lie group $G$. We prove

\begin{theorem}\label{main-thm-4}
For a left invariant complex Finsler metric on a complex Lie group $G$, its complex spray $\chi=w^i\delta_{z^i}$ can be extended to the holomorphic tangent field $\chi=u^i\widetilde{U}^i=v^i\widetilde{V}^i$ on $T^{1,0}G$. Further more, when $\chi$ is viewed as a real
tangent field on $TG$, it coincides the canonical bi-invariant spray structure on $G$.
\end{theorem}

From Theorem \ref{main-thm-4},
we see another major difference between real and complex homogeneous Finsler geometry. In the context of Theorem \ref{main-thm-4}, we can not find any information of the metric or define a useful spray vector field on the Lie algebra, as in \cite{Hu2015}\cite{Hu2015-2}, from the complex spray.

Next, we consider the three different K\"{a}hler properties in complex Finsler geometry and prove

\begin{theorem}\label{main-thm-2}
For a left invariant complex Finsler metric $F$ on a complex Lie group $G$, the strongly K\"{a}hler,  K\"{a}hler, and weakly K\"{a}hler properties are all equivalent. More over, $F$ is K\"{a}hler if and only if $G$ has an Abelian Lie algebra.
\end{theorem}

From Theorem \ref{main-thm-2}, we see that the K\"{a}hler property here only depends on the algebraic property of $G$ and totally irrelevant to the metric itself.

Finally, we calculate the holomorphic sectional curvature and prove

\begin{theorem}\label{main-thm-3}
A left invariant complex Finsler metric on a complex Lie group has vanishing holomorphic sectional curvature everywhere.
\end{theorem}

Theorem \ref{main-thm-3} justifies the folklore that many new examples of metrics with constant holomorphic sectional curvature can be found in Finsler geometry.

\begin{remark} C.P. Zhong and his coauthors, W. Xiao and K.K. Luo, have independently proved some main results in this paper. They first noticed the vanishing of the holomorphic sectional curvature for left invariant complex Finsler metrics on complex Lie groups. Their method is different from ours and their manuscript is now in preparation.
\end{remark}

To summarize, complex homogeneous Finsler geometry behave quite differently from its real analog. The special situation we have considered, i.e.,
a left invariant complex Finsler metric on a complex Lie group,
is a relatively simple case. However, we still believe that the technique of invariant frame is a useful tool, which may help us explore more general complex homogeneous Finsler manifolds.

This paper is scheduled as follows. In Section 2, we collect some necessary knowledge in complex geometry and complex Finsler geometry. In Section 3, we construct the complex invariant frames on $T^{1,0}G$ and prove its properties. In Section 4, we prove above main theorems. In Section 5, we provide an appendix for the lift of a smooth real tangent field.
\section{Preliminaries}

\subsection{Complex manifold and complex vector field}

Let $M$ be a {\it complex manifold} of complex dimension $n>0$. We use $z=(z^1,\cdots,z^n)\in M$ to denote a (holomorphic) local coordinate
on $M$.
For a complex function $f$, we present its real and imaginary part by $f_\Re$
and $f_\Im$ respectively, i.e., $f=f_\Re+\sqrt{-1}f_\Im$. So for
a local coordinate $(z^1,\cdots,z^n)\in M$,
$(z^1_\Re,z^1_\Im,\cdots,z^n_\Re,z^n_\Im)$
is the corresponding real local coordinate.

Let $J:T_zM\rightarrow T_zM$ be the complex structure at $z\in M$, where $T_zM$ is the real tangent space. Then $T_zM\otimes \mathbb{C}=T^{1,0}_zM\oplus T^{0,1}_zM$, in which the two summands are the eigenspaces for $\pm\sqrt{-1}$ respectively.
In local coordinate, $J\partial_{z^i_\Re}=\partial_{z^i_\Im}$ and $J\partial_{z^i_\Im}=-\partial_{z^i_\Re}$,
so we have
\begin{eqnarray}
T^{0,1}_zM &=& \mathrm{span}^\mathbb{C}
\{\partial_{z^i}=\tfrac12(\partial_{z^i_\Re}-\sqrt{-1}\partial_{z^i_\Im}),\forall i\}\quad\mbox{and}\nonumber\\
T^{1,0}_zM &=& \mathrm{span}^\mathbb{C}\{
\partial_{\overline{z}^i}=\tfrac12(\partial_{z^i_\Re}+
\sqrt{-1}\partial_{z^i_\Im}),\forall i\}.\label{004}
\end{eqnarray}

The {\it holomorphic tangent bundle} $T^{1,0}M=\bigcup_{z\in M} T^{1,0}_zM$ is a holomorphic bundle over $M$.
A section of $T^{1,0}M$ is {\it complex tangent field} $V$ {\it of type} $(1,0)$, which can be locally presented as
$V=f^i\partial_{z^i}$, in which $f^i$ are complex function in the local chart.
In this paper, we only consider smooth tangent field of type $(1,0)$, so to avoid iteration, we will not mention the smoothness or the type in later discussion.
The smooth isomorphism
 between real vector bundles, $\cdot^\circ: T^{1,0}M\rightarrow TM$,
$\partial_{z^i}{}^\circ=\partial_{z^i_\Re}$,
$(\sqrt{-1}\partial_{z^i})^{\circ}=\partial_{z^i_\Im}$,
provides a one-to-one correspondence between complex tangent vector fields
and real tangent vector fields, i.e., for $V=f^i\partial_{z^i}$, $V^\circ=V+\overline{V}=f^i_\Re\partial_{z^i_\Re}+f^i_\Im\partial_{z^i_\Im}$,
and conversely, $V=\tfrac12(V^\circ-\sqrt{-1}J(V^\circ))$.
In later discussion, $T^{1,0}M$ is occasionally viewed as $TM$ through $\cdot^\circ$, complex tangent fields are usually denoted as $U$, $V$, $W$, etc., and real tangent fields as $U^\circ$, $V^\circ$, $W^\circ$, etc.

For two complex
tangent fields $U=f^i\partial_{z^i}$ and $V=g^i\partial_{z^i}$,
their bracket is
$[U,V]=(f^i\partial_{z^i}g^j-g^i\partial_{z^i}f^j)\partial_{z^j}$. Generally speaking, $[U,V]^\circ$ does not coincide with
\begin{eqnarray}
[U^\circ,V^\circ]&=&(
f^i_\Re\partial_{z^i_\Re}g^j_\Re+f^i_\Im\partial_{z^i_\Im}g^j_\Re
-g^i_\Re\partial_{z^i_\Re}f^j_\Re-g^i_\Im\partial_{z^i_\Im}f^j_\Re
)\partial_{z^j_\Re}\nonumber\\
& &+
(
f^i_\Re\partial_{z^i_\Re}g^j_\Im+f^i_\Im\partial_{z^i_\Im}g^j_\Im
-g^i_\Re\partial_{z^i_\Re}f^j_\Im-g^i_\Im\partial_{z^i_\Im}f^j_\Im
)\partial_{z^j_\Re},\label{001}
\end{eqnarray}
because, generally speaking, $[U^\circ,V^\circ]\neq[U,V]^\circ$. On the other hand, when $U$ and $V$ are {\it holomorphic}, i.e. all $f^i$ and $g^i$ are holomorphic,
we have
\begin{lemma}\label{lemma-1}
For two holomorphic tangent fields $U$ and $V$,
$[U^\circ,V^\circ]=[U,V]^\circ$. Further more, we also have $[J(U^\circ),V^\circ]=[U^\circ,J(V^\circ)]=J([U,V]^\circ)$.
\end{lemma}
\begin{proof}
By the Cauchy-Riemann equation, i.e.,
\begin{eqnarray*}
\partial_{z^i_\Re}f^j_\Re=\partial_{z^i_\Im}f^j_\Im,\quad
\partial_{z^i_\Re}f^j_\Im=-\partial_{z^i_\Im}f^j_\Re,\quad
\partial_{z^i_\Re}g^j_\Re=\partial_{z^i_\Im}g^j_\Im,\quad
\partial_{z^i_\Re}g^j_\Im=-\partial_{z^i_\Im}g^j_\Re,\quad\forall i,j,
\end{eqnarray*}
 (\ref{001}) can be reformulated as
\begin{eqnarray}\label{002}
[U^\circ,V^\circ]&=&(
f^i_\Re\partial_{z^i_\Re}g^j_\Re-f^i_\Im\partial_{z^i_\Re}g^j_\Im
-g^i_\Re\partial_{z^i_\Re}f^j_\Re+g^i_\Im\partial_{z^i_\Re}f^j_\Im
)\partial_{z^j_\Re}\nonumber\\
& &+
(
f^i_\Re\partial_{z^i_\Re}g^j_\Im+f^i_\Im\partial_{z^i_\Re}g^j_\Re
-g^i_\Re\partial_{z^i_\Re}f^j_\Im-g^i_\Im\partial_{z^i_\Re}f^j_\Re
)\partial_{z^j_\Re}.
\end{eqnarray}
On the other hand, the Cauchy-Riemann equation implies
 implies
$$\partial_{z^i}f^j=\partial_{z^i_\Re}f^j=-\sqrt{-1}\partial_{z^i_\Im}f^j,\quad
\partial_{z^i}g^j=\partial_{z^i_\Re}g^j=-\sqrt{-1}\partial_{z^i_\Im}g^j,
\quad\forall i,j.$$
So we have
\begin{eqnarray}\label{003}
[U,V]^\circ&=&((f^i_\Re+\sqrt{-1}f^i_\Im)(\partial_{z^i_\Re}g^j_\Re+
\sqrt{-1}\partial_{z^i_\Re}g^j_\Im)-(g^i_\Re+\sqrt{-1}g^i_{\Im})(\partial_{z^i_\Re}f^j_\Re
+\sqrt{-1}\partial_{z^i_\Re}f^j_\Im)\partial_{z^j}\nonumber\\
&=& (f^i_\Re\partial_{z^i_\Re}g^j_\Re-f^i_\Im\partial_{z^i_\Re}g^j_\Im-
g^i_\Re\partial_{z^i_\Re}f^j_\Re+g^i_\Im\partial_{z^i_\Re}f^j_\Im)
\partial_{z^j_\Re}
\nonumber\\
& &+(f^i_\Re\partial_{z^i_\Re}g^j_\Im+f^i_\Im\partial_{z^i_\Re}g^j_\Re
-g^i_\Re\partial_{z^i_\Re}f^j_\Im-g^i_\Im\partial_{z^i_\Re}f^j_\Re)
\partial_{z^j_\Im}.
\end{eqnarray}
Comparing (\ref{002}) and (\ref{003}), the first statement of Lemma \ref{lemma-1} is proved.

Since $U$ is holomorphic, $\sqrt{-1}U$ is also holomorphic.
We have $(\sqrt{-1}U)^\circ=J(U^\circ)$, so
$$[J(U^\circ),V^\circ]=[(\sqrt{-1}U)^\circ,V^\circ]=
[\sqrt{-1}U,V]^\circ=(\sqrt{-1}[U,V])^\circ=J([U,V]^\circ).$$
The proof of $[U^\circ,J(V^\circ)]=J([U,V]^\circ)$ is similar.
\end{proof}

\subsection{Complete lifting of a holomorphic vector field}

For a complex manifold $M$, the holomorphic tangent bundle $T^{1,0}M$
is also a complex manifold. For the local coordinate $z=(z^1,\cdots,z^n)$ on $M$, we have the standard local coordinate
$(z,w)=(z^1,\cdots,z^n,w^1,\cdots,w^n)$ and the corresponding real local coordinate $(z^1_\Re,z^1_\Im,\cdots,z^n_\Re,z^n_\Im,w^1_\Re,w^1_\Im,\cdots,w^n_\Re,w^n_\Im)$
on $T^{1,0}M$, where $w=w^i\partial_{z^i}\in T^{1,0}_zM$ is also viewed as $w^\circ=w^i_\Re\partial_{z^i_\Re}+w^i_\Im\partial_{z^i_\Im}\in T_zM$.
The complex structure $J$ on $TM$ is given by
$$J\partial_{z^i_\Re}=\partial_{z^i_\Im},\quad
J\partial_{z^i_\Im}=-\partial_{z^i_\Re},\quad
J\partial_{w^i_\Re}=\partial_{w^i_\Im},\quad
J\partial_{w^i_\Im}=-\partial_{w^i_\Re},\quad\forall i,$$
so we have
\begin{eqnarray}
T^{1,0}_{(z,w)}(T^{1,0}M)=\mathrm{span}^\mathbb{C}\{
 \partial_{z^i}=\tfrac12(\partial_{z^i_\Re}-\sqrt{-1}\partial_{z^i_\Im}),
 \partial_{w^i}=\tfrac12(\partial_{w^i_\Re}-\sqrt{-1}\partial_{w^i_\Im}),
\forall i\}.
\label{005}
\end{eqnarray}

Now we introduce the lift of a complex tangent field $V$ from $M$ to $T^{1,0}M$. Around each $z\in M$ where $V$ is defined, $V^\circ$ is a real tangent field, which can be lifted to a real tangent field $\widetilde{V^\circ}$ on $TM$. To be precise, $V^\circ$
generates a family of local diffeomorphisms $\rho_t$. The tangent maps
$(\rho_t)_*$ are a family of local diffeomphisms on $TM$. The lift $\widetilde{V^\circ}$ from $M$ to $TM$ is given by $\tfrac{{\rm d}}{{\rm d}t}|_{t=0}(\rho_t)_*$. Using the canonical identification $\cdot^\circ:T^{1,0}M\rightarrow TM$, $\widetilde{V^\circ}$ can be viewed as a real tangent field on $T^{1,0}M$. Then it determines a complex tangent field
$\widetilde{V}$ on $T^{1,0}M$ by $(\widetilde{V})^\circ=\widetilde{V^\circ}$.
We call this $\widetilde{V}$ the {\it lift of $V$ from $M$ to $T^{1,0}M$}.

\begin{lemma}
\label{lemma-2}
For a holomorphic tangent field on $M$ locally presented as $V=f^i\partial_{z^i}$, its lift from $M$ to $T^{1,0}M$ is also a holomorphic tangent field, which can be locally presented as $\widetilde{V}=f^i\partial_{z^i}+w^i\partial_{z^i}f^j\partial_{w^j}$.
\end{lemma}
\begin{proof}
Since $V^\circ=f^i_\Re\partial_{z^i_\Re}+f^i_\Im\partial_{z^i_\Im}$,
Lemma \ref{lemma-9} in Appendix tells us
\begin{eqnarray}
& &\widetilde{{V}^\circ}=f^i_\Re\partial_{z^i_\Re}+f^i_\Im\partial_{z^i_\Im}
+w^i_\Re\partial_{z^i_\Re}f^j_\Re\partial_{w^j_\Re}
+w^i_\Re\partial_{z^i_\Re}f^j_\Im\partial_{w^j_\Im}
+w^i_\Im\partial_{z^i_\Im}f^j_\Re\partial_{w^j_\Re}
+w^i_\Im\partial_{z^i_\Im}f^j_\Im\partial_{w^j_\Im}\nonumber\\
&=&
f^i_\Re\partial_{z^i_\Re}+f^i_\Im\partial_{z^i_\Im}
+w^i_\Re\partial_{z^i_\Re}f^j_\Re\partial_{w^j_\Re}
+w^i_\Re\partial_{z^i_\Re}f^j_\Im\partial_{w^j_\Im}
-w^i_\Im\partial_{z^i_\Re}f^j_\Im\partial_{w^j_\Re}
+w^i_\Im\partial_{z^i_\Re}f^j_\Re\partial_{w^j_\Im},\label{007}
\end{eqnarray}
in which we have used the holomorphic property, i.e., $\partial_{z^i_\Re}f^j_\Re=\partial_{z^i_\Im}f^j_\Im$ and
$\partial_{z^i_\Re}f^j_\Im=\partial_{z^i_\Im}f^j_\Re$ for all $i$ and $j$, for the second equality.
On the other hand, for $f^i\partial_{z^i}+w^i\partial_{z^i}f^j\partial_{w^j}$ on $T^{1,0}M$, we have
\begin{eqnarray}
& &(f^i\partial_{z^i}+w^i\partial_{z^i}f^j\partial_{w^j})^\circ
=((f^i_\Re+\sqrt{-1}f^i_\Im)\partial_{z^i}
+(w^i_\Re+\sqrt{-1}w^i_\Im)(\partial_{z^i_\Re}f^j_\Re+
\sqrt{-1}\partial_{z^i_\Re}f^j_\Im)\partial_{w^j})^\circ\nonumber\\
&=&f^i_\Re\partial_{z^i_\Re}+f^i_\Im\partial_{z^i_\Im}
+w^i_\Re\partial_{z^i_\Re}f^j_\Re\partial_{w^j_\Re}
-w^i_\Im\partial_{z^i_\Re}f^j_\Im\partial_{w^j_\Re}
+w^i_\Re\partial_{z^i_\Re}f^j_\Im\partial_{w^j_\Im}
+w^i_\Im\partial_{z^i_\Re}f^j_\Re\partial_{w^j_\Im},\label{008}
\end{eqnarray}
in which we have used the holomorphic property, i.e., $\partial_{z^i}f^j=\partial_{z^i_\Re}f^j$ for all $i$ and $j$, for the first equality. Compare (\ref{007}) and (\ref{008}), we get
$\widetilde{V^\circ}=(f^i\partial_{z^i}+w^i\partial_{z^i}f^j\partial_{w^j})^\circ$
on $T^{1,0}M$, so $\widetilde{V}=f^i\partial_{z^i}+w^i\partial_{z^i}f^j\partial_{w^j}$, which  is obviously holomorphic.
\end{proof}

\begin{lemma}\label{lemma-7}
For a holomorphic tangent field $V$ on $M$, we have $J(\widetilde{V}^\circ)=\widetilde{J(V^\circ)}=(\widetilde{\sqrt{-1} V})^\circ$.
\end{lemma}
\begin{proof}
For the local presentation $V=f^i\partial_{z^i}$, $V^\circ=f^i_\Re\partial_{z^i_\Re}+f^i_\Im\partial_{z^i_\Im}$ and
$J(V^\circ)=-f^i_\Im\partial_{z^i_\Re}+f^i_\Re\partial_{z^i_\Im}$, then we have
\begin{eqnarray*}
& &\widetilde{J(V^\circ)}
=-f^i_\Im\partial_{z^i_\Re}+f^i_\Re\partial_{z^i_\Im}
-w^i_\Re\partial_{z^i_\Re}f^j_\Im\partial_{w^j_\Re}
-w^i_\Im\partial_{z^i_\Im}f^j_\Im\partial_{w^j_\Re}
+w^i_\Re\partial_{z^i_\Re}f^j_\Re\partial_{w^j_\Im}
+w^i_\Im\partial_{z^i_\Im}f^j_\Re\partial_{w^j_\Im}\\
&=&-f^i_\Im\partial_{z^i_\Re}+f^i_\Re\partial_{z^i_\Im}
-w^i_\Re\partial_{z^i_\Re}f^j_\Im\partial_{w^j_\Re}
-w^i_\Im\partial_{z^i_\Re}f^j_\Re\partial_{w^j_\Re}
+w^i_\Re\partial_{z^i_\Re}f^j_\Re\partial_{w^j_\Im}
-w^i_\Im\partial_{z^i_\Re}f^j_\Im\partial_{w^j_\Im}\\
&=&J(f^i_\Re\partial_{z^i_\Re}+f^i_\Im\partial_{z^i_\Im}
+w^i_\Re\partial_{z^i_\Re}f^j_\Im\partial_{w^j_\Im}
+w^i_\Im\partial_{z^i_\Re}f^j_\Re\partial_{w^j_\Im}
+w^i_\Re\partial_{z^i_\Re}f^j_\Re\partial_{w^j_\Re}
-w^i_\Im\partial_{z^i_\Re}f^j_\Im\partial_{w^j_\Re}),
\end{eqnarray*}
where we have applied the Cauchy-Riemann equation for the second equality.
Compare it with (\ref{008}), then we see $J(\widetilde{V}^\circ)=\widetilde{J(V^\circ)}$.
The proof of $(\widetilde{\sqrt{-1} V})^\circ=J(\widetilde{V}^\circ)$ is similar.
\end{proof}

Let $\{V_i,\forall i\}$ be a family of   holomorphic tangent fields on $M$. We call it a {\it holomorphic frame} if for each possible $z\in M$, $\{V_i(z),\forall i\}$ is a complex basis of $T^{1,0}M$. In $T^{1,0}_zM$, we have complex linear coordinate functions $v^i$ and the corresponding $\partial_{v^i}$, such that
each vector of $T^{1,0}_zM$ is presented as $v^i V_i(z)$. Then $\{\widetilde{V}_i,\partial_{v^i},\forall i\}$ is
a holomorphic frame on $T^{1,0}M$, for which we can generalize Lemma \ref{lemma-2} as follows.

\begin{lemma}\label{lemma-6}
For the holomorphic frame $\{\widetilde{V}_i,\partial_{v^i},\forall i\}$ and
the holomorphic tangent field $U=f^iV_i$, we have $\widetilde{U}=f^i\widetilde{V}_i+v^i V_i(f^j)\partial_{v^j}$.
\end{lemma}
\begin{proof}
For the local coordinate $(z,w)=(z^1,\cdots,z^n,w^i,\cdots,w^n)$ on $T^{1,0}M$, we have $GL(n,\mathbb{C})$-valued holomorphic functions $A^i_j=A^i_j(z)$ and $(B^i_j)=(B^i_j(z))$ with $(B^i_j)=(A^i_j)^{-1}$ (i.e., $A^i_jB^j_k=B^i_jA^j_k=\delta^i_{k}$), satisfying
$V_j=A^i_j\partial_{z^i}$ and $\partial_{z^j}=B^i_j V_i$. More over,
we have
$w^i=A^i_jv^j$, $v^i=B^i_jw^j$, $\partial_{w^j}=B^i_j\partial_{v^i}$ and $\partial_{v^j}=A^i_j\partial_{w^i}$.
Then Lemma \ref{lemma-2} provides
$\widetilde{V}_j=A^i_j\partial_{z^i}+w^i\partial_{z^i}A^k_j\partial_{w^k}$, for $U=f^iV_i=f^iA_i^j\partial_{z^j}$, calculation shows
\begin{eqnarray*}
\widetilde{U}&=&f^iA_i^j\partial_{z^j}+
w^i\partial_{z^i}(f^jA^k_j)\partial_{w^k}
=f^iA_i^j\partial_{z^j}+w^i f^j\partial_{z^i}A^k_j\partial_{w^k}
+w^i A^k_j \partial_{z^i}f^j\partial_{w^k} \\
&=&f^i(A_i^j\partial_{z^j}+w^j\partial_{z^j}A^k_i\partial_{w^k})+
v^p A_p^i B_i^q V_q(f^j)\partial_{v^j}
=f^i\widetilde{V}_i+v^i V_i(f^j)\partial_{v^j}.
\end{eqnarray*}
The proof of Lemma \ref{lemma-6} is finished.
\end{proof}

We end this subsection with the following lemma.
\begin{lemma}\label{lemma-11}
For any two holomorphic tangent fields $U$ and $V$, $\widetilde{[U,V]}=[\widetilde{U},\widetilde{V}]$.
\end{lemma}
\begin{proof}
Using Lemma \ref{lemma-1}, Lemma \ref{lemma-2} and Lemma \ref{lemma-10}, we can get
\begin{eqnarray}
(\widetilde{[U,V]})^\circ=\widetilde{[U,V]^\circ}=\widetilde{[U^\circ,V^\circ]}
=[\widetilde{U^\circ},\widetilde{V^\circ}].\label{013}
\end{eqnarray}
On the other hand,
by Lemma \ref{lemma-1} and Lemma \ref{lemma-7}, $[J(\widetilde{U}^\circ),J(\widetilde{V}^\circ)]
=-[\widetilde{U}^\circ,\widetilde{V}^\circ]$, so the real part of
$$[\widetilde{U},\widetilde{V}]=
[\frac12(\widetilde{U}^\circ-\sqrt{-1}J\widetilde{U}^\circ),
\frac12(\widetilde{V}^\circ-\sqrt{-1}J\widetilde{V}^\circ)]$$ coincides with
$
\frac14[\widetilde{U}^\circ,\widetilde{V}^\circ]-\frac14[J\widetilde{U}^\circ,
J\widetilde{U}^\circ]=
\frac12[\widetilde{U}^\circ,\widetilde{V}^\circ]$,
so we have $[\widetilde{U},\widetilde{V}]^\circ=[\widetilde{U}^\circ,\widetilde{V}^\circ]$.
Combine it with (\ref{013}), we finish the proof of Lemma \ref{lemma-11}.
\end{proof}
\subsection{Complex Finsler metric}

A {\it complex Finsler metric} on the complex manifold $M$
is a continuous function $F: T^{1,0}M\rightarrow[0,+\infty)$ satisfying
the following properties:
\begin{enumerate}
\item $F$ is regular, i.e., it is positive and smooth when restricted to $T^{1,0}M\backslash0$;
\item $F$ is 1-homogeneous, i.e., $F(z,\lambda w)=|\lambda|F(z,w)$ for any $(z,w)\in T^{1,0}M$
and $\lambda\in\mathbb{C}$.
\item $F$ is pseudo-convex, i.e., for $\G=F^2$,
the complex Hessian matrix $(\G_{i\overline{j}})=(\partial_{w^i}\partial_{\overline{w}^j}\G)$
is positive definite at any $(z,w)\in T^{1,0}M\backslash 0$.
\end{enumerate}
The restriction of the complex Finsler metric $F$ to
each $T^{1,0}_xM$ is called a {\it complex Minkowski norm}, which can be similarly defined for any finite dimensional complex linear space.

Denote by $(\G^{\overline{i}j})$ the inverse matrix of $(\G_{\overline{i}j})$, i.e., $\G_{k\overline{j}}\G^{\overline{j}i}=
\G^{\overline{i}j}\G_{j\overline{k}}=\delta^i_k$. We apply the convention
abbreviating $\partial_{w^i}\G$,
$\partial_{w^i}\partial_{\overline{w}^j}\partial_{\overline{w}^k}\G$,
$\partial_{\overline{w}^j}\partial_{z^k}\G$, etc., as
$\G_i$, $\G_{i\overline{j}\overline{k}}$, $\G_{\overline{j};k}$, etc., respectively.
The homogeneity property of $F$ implies $\G=\G_iw^i=\G_{\overline{i}}\overline{w}^i$,
$\G_{ij}w^i=\G_{i\overline{j}k}w^i=\G_{i\overline{j}\overline{k}}w^j=0$, etc.,
everywhere on $T^{1,0}M\backslash0$.

The coefficients for the Chern-Finsler connection are given by
$$N^i_k=\G^{\overline{j}i}\G_{\overline{j};k}=\G^{\overline{j}i}
\partial_{\overline{w}^j}\partial_{z^k}F^2,\quad
\Gamma^i_{j;k}=\partial_{{w}^j}N^i_k\quad\mbox{and}\quad
C^i_{jk}=\G^{\overline{l}i}\G_{j {k}\overline{l}}=
\G^{\overline{l}i}\partial_{w^k}\G_{j\overline{l}}.$$
We call $F$ {\it complex Berwald} if all $\Gamma^i_{j;k}$ are independent of the $w$-coordinate. We call $F$ {\it strongly K\"{a}hler} ({\it K\"{a}hler} or {\it weakly K\"{a}hler})
if
$$\Gamma^i_{j;k}-\Gamma^i_{k;j}=0 \ \mbox{(}w^k(\Gamma^i_{j;k}-\Gamma^i_{k;j})=0\mbox{ or }w^k(\Gamma^i_{j;k} -\G_i\Gamma^i_{k;j})\G_i=0\mbox{ respectively})$$
is satisfied everywhere on $T^{1,0}M\backslash0$.

Denote by $\delta_{z^i}=\partial_{z^i}-N_i^j\partial_{w^j}$ the complex tangent fields on $T^{1,0}M$, which linearly span the horizonal distribution. Then $\chi=w^i\delta_{z^i}$ is a complex tangent field, which can be globally defined on $T^{1,0}M\backslash0$. It is the analog for the geodesic spray in real Finsler geometry, so we call it the {\it complex spray}
for $(M,F)$.

The {\it holomorphic sectional curvature} can be presented as the function
$$K=2R_{i\overline{j}k\overline{l}}\frac{w^i\overline{w}^j
w^k\overline{w}^l}{\G^2}$$
on $T^{1,0}M\backslash0$, in which
$R_{i\overline{j}k\overline{l}}=-\G_{i\overline{j};k\overline{l}}+
\G^{\overline{q}p}\G_{i\overline{q};k}\G_{p\overline{j};\overline{l}}
=-\partial_{z^k}\partial_{\overline{z}^l}\G_{i\overline{j}}+
\G^{\overline{q}p}\partial_{z^k}\G_{i\overline{q}}\partial_{\overline{z}^l}
\G_{p\overline{j}}$.

More details can be found in \cite{Ab1994}.
\section{Left invariant complex Finsler metrics and invariant frames}

\subsection{Left invariant complex Finsler metric on a complex Lie group}
Let $G$ be a
{\it complex Lie group}. By definition it means that, $G$ has a group structure with the identity element $e\in G$ specified, such that the multiplication map $(g_1,g_2)\in G\times G\mapsto g_1g_2\in G$ and the inverse map
$g\in G\mapsto g^{-1}\in G$ are holomorphic.
For any $g\in G$, Denote by $L_g(g')=gg'$ and $R_g(g')=g'g$ with $g'\in G$ the left and right translations  respectively. The holomorphic property of the multiplication map implies that all left and right translations
are holomorphic diffeomorphisms on $G$.

A complex tangent field $V$ on $G$ is called {\it left invariant} ({\it right invariant}), if $V^\circ$ is left invariant (right invariant), i.e., $V^\circ$ is preserved by the tangent maps of all left translations (all right translations respectively).

\begin{lemma}\label{lemma-3}
Let $V^\circ$ be a left or right invariant real tangent field on $G$, then $V$ is holomorphic.
\end{lemma}

\begin{proof}We only need to verify the lemma for a left invariant $V^\circ$,
because the proof in other case is similar. Let $R_{g(t)}$ be the one-parameter subgroup of right translations generated by $V^\circ$. Using a local chart,
we have a smooth map
\begin{eqnarray*}
& &(g',t)=(z^1,\cdots,z^n,t)=(z^1_\Re,z^1_\Im,\cdots,z^n_\Re,z^n_\Im,t)\in G\times\mathbb{R}\\
&\mapsto& R_{g(t)}g'=(f^1(z,t),\cdots,f^n(z,t))=
(f^1_\Re,f^1_\Im,\cdots,f^n_\Re,f^n_\Im),
\end{eqnarray*}
in which all $f^i$ are holomorphic with respect to $z$.
Then
$V^\circ=(\partial_{t}|_{t=0}f^i_\Re)\partial_{z^i_\Re}+
(\partial_{t}|_{t=0}f^i_\Im)\partial_{z^i_\Im}$ and $V=(\partial_t|_{t=0} f^i)\partial_{z^i}$.
Apply $\partial_t|_{t=0}$ to the Cauchy-Riemann equations $$\partial_{z^i_\Re}f^j_\Re=\partial_{z^i_\Im}f^j_\Im,\quad \partial_{z^i_\Im}f^j_\Re=-\partial_{z^i_\Re}f^j_\Im,\quad\forall i,j,$$
we get
$$\partial_{z^i_\Re}(\partial_t|_{t=0}f^j_\Re)=\partial_{z^i_\Im}(\partial_t|_{t=0}f^j_\Im),
\quad \partial_{z^i_\Im}(\partial_t|_{t=0}f^j_\Re)=
-\partial_{z^i_\Re}(\partial_t|_{t=0}f^j_\Im),\quad\forall i,j.$$
So $V$ is holomorphic with respect to $z$, which ends the proof of Lemma \ref{lemma-3}.
\end{proof}

\begin{lemma}\label{lemma-4}
Let $J:TG\rightarrow TG$ be the complex structure on $G$. Then for any left (right) invariant complex tangent field $V$ on $G$, $\sqrt{-1}V$ is also a
left (right) invariant complex tangent field, i.e.,
$J(V^\circ)=(\sqrt{-1}V)^\circ$ is a left (right respectively) invariant real tangent field.
\end{lemma}

\begin{proof}
We only need to verify the lemma when $V$ is left invariant, because the proof for the other half is similar. The left invariance of $V^\circ$ means
$(L_g)_* (V^\circ)=V^\circ$ for any $g\in G$.
Since $L_g:G\rightarrow G$
is holomorphic, its tangent map $(L_g)_*:TG\rightarrow TG$ commutes with $J$. So we have $(L_g)_*(J(V^\circ))=J((L_g)_*(V^\circ))=J(V^\circ)$ for any $g\in G$, which implies $J(V^\circ)$ is left invariant. The proof of Lemma \ref{lemma-4} is finished.
\end{proof}

Let $F:T^{1,0}M\rightarrow[0,\infty)$ be a complex Finsler metric on $G$. We call $F$ {\it left invariant}, if it is preserved by the tangent maps of all left translations $(L_g)_*:T^{1,0}G=TG\rightarrow TG=T^{1,0}G$. Here we identify $T^{1,0}$ with $TG$ through $\cdot^\circ$.

\begin{lemma}\label{lemma-5}
For any right invariant complex tangent field $V$ on $G$, we have $\widetilde{V}\G\equiv0$.
\end{lemma}
\begin{proof} Since the right invariant real tangent field $V^\circ$ generates
left translations on $G$, the left invariance of $G$ implies $\widetilde{V^\circ} \G\equiv0$. By Lemma \ref{lemma-4}, $J(V^\circ)$ is also right invariant. So Lemma \ref{lemma-7} implies  $(J(\widetilde{V}^\circ))\G = \widetilde{J(V^\circ)}\G\equiv0$. To summarize, $\widetilde{V}\G=
\tfrac12(\widetilde{V}^\circ\G-\sqrt{-1}(J(\widetilde{V}^\circ))\G)
\equiv0$, which ends the proof.
\end{proof}
\subsection{Technique of invariant frames}
The Lie algebra $\mathfrak{g}$ of a complex Lie group $G$
is a complex Lie algebra, which can be identified with the space of all left invariant complex tangent fields, with the canonical bracket between complex tangent fields. It can also be identified with $T^{1,0}_eG$.

We choose a complex basis $\{e_1,\cdots,e_n\}$ of
$\mathfrak{g}=T^{1,0}G$, with the corresponding bracket coefficients $c^k_{ij}$ determined by $[e_i,e_j]=c_{ij}^ke_k$. Then we have left invariant complex tangent fields $\{U_i,\forall i\}$ and right invariant complex tangent fields $\{V_i,\forall i\}$, satisfying $U_i(e)=V_i(e)=e_i$ for each $i$. Denote $\widetilde{U}_i$ and $\widetilde{V}_i$ their lift to $T^{1,0}G$ respectively.
Any $w\in T^{1,0}_gG$ can be uniquely presented as $w=u^iU_i(g)=v^iV_i(g)$, so we have the complex linear coordinates $(u^1,\cdots,u^n)$ and $(v^1,\cdots,v^n)$, and the corresponding  $\{\partial_{u^i},\forall i\}$, $\{\partial_{v^i},\forall i\}$ in $T^{1,0}_gG$, which can be globally defined
as holomorphic functions and tangent fields on $T^{1,0}G$. To summarize, we get the {\it left invariant
holomorphic frame} $\{\widetilde{U}_i,\partial_{u^i},\forall i\}$ and the
{\it right invariant holomorphic frame}
$\{\widetilde{V}_i,\partial_{v^i},\forall i\}$.

The interrelation between these two frames are as follows.
For each $g\in G$, we define $\phi^i_j=\phi^i_j(g)$ and $\psi^i_j=\psi^i_j(g)$, such that
$\mathrm{Ad}(g)e_j=\phi^i_j e_i$ and $\mathrm{Ad}(g^{-1})e_i=\psi_j^i e_i$
(so we also have $(\psi^i_j)=(\phi^i_j)^{-1}$, i.e., $\psi^i_j\phi^j_k=\phi^i_j\psi^j_k=\delta^i_k$).

\begin{lemma}\label{lemma-8}
(1) $U_j=\phi^i_j V_i$, $V_j=\psi^i_j U_i$, $u^i=\psi^i_j v^j$,
$v^i=\phi^i_j u^j$, $\partial_{u^j}=\phi^i_j\partial_{v^i}$,
$\partial_{v^j}=\psi^i_j\partial_{u^i}$;

(2) $[U_i,U_j]=c^k_{ij}U_k$, $[V_i,V_j]=-c^k_{ij}V_k$,
$[\widetilde{U}_i,\widetilde{V}_j]=0$,
$[\widetilde{U}_i,\widetilde{U}_j]=c^k_{ij}\widetilde{U}_k$,
$[\widetilde{V}_i,\widetilde{V}_j]=-c^k_{ij}V_k$;

(3) $U_i(\psi_j^k)=\psi_j^lc^k_{li}=\phi_i^l\psi_p^k c_{jl}^p$,
$V_i(\phi_j^k)=\phi_j^lc^k_{il}=\phi_p^k\psi_i^l c_{lj}^p$;

(4) $\widetilde{U}_i=\phi^j_i\widetilde{V}_j+c^k_{li}u^l\partial_{u^k}$,
$\widetilde{V}_i=\psi^j_i\widetilde{U}_j-c^k_{li}v^l\partial_{v^k}$.

(5) $\widetilde{U}_j v^i=\widetilde{V}_j u^i=0$,
$\widetilde{U}_i u^j=c_{li}^j u^l$,
$\widetilde{V}_i v^j=-c_{li}^j v^l$;

(6) $[\widetilde{U}_i,\partial_{u^j}]=c^k_{ij}\partial_{u^k}$,
$[\widetilde{U}_i,\partial_{v^j}]=0$, $[\widetilde{V}_i,\partial_{u^j}]=0$,
$[\widetilde{V}_i, \partial_{v^j}]=-c^k_{ij}\partial_{v^k}$.
\end{lemma}

\begin{proof}Some statements in Lemma \ref{lemma-8} can be proved as follows. The others can be proved similarly or follows easily.

(1) For any $g\in G$,
$$(L_{g^{-1}})_* (V_j(g))=(L_{g^{-1}})_*\circ (R_g)_*(e_j)=\mathrm{Ad}(g^{-1})e_j=\psi_j^i e_i,$$
so we have $V_j=\psi_j^i U_i$.

(2)
By Lemma \ref{lemma-4}, $\{\widetilde{U^\circ_i},\widetilde{J(U^\circ_i)},\partial_{u^i_\Re},
\partial_{u^i_\Im},\forall i\}$ and
and $\{\widetilde{V^\circ_i},\widetilde{J(V^\circ_i)},\partial_{v^i_\Re},
\partial_{v^i_\Im}\}$ are left and right invariant frame on $TG$  respectively, in the real context \cite{Xu2022-1}. They are induced by the real basis $\{e_1,\cdots,e_n,\sqrt{-1}e_1,\cdots$, $\sqrt{-1}e_n\}$ of $\mathfrak{g}$, with the following bracket coefficients,
\begin{eqnarray*}
& &[e_i,e_j]=-[\sqrt{-1}e_i,\sqrt{-1}e_j]=\mathrm{Re}c_{ij}^k e_k+\mathrm{Im}c_{ij}^k \sqrt{-1}e_k, \\
& &[\sqrt{-1}e_i,e_j]=[e_i,\sqrt{-1}e_j]=-\mathrm{Im}c_{ij}^k
e_k+\mathrm{Re}c_{ij}^k \sqrt{-1}e_k.
\end{eqnarray*}
So we have
\begin{eqnarray*}
& &[U^\circ_i,U^\circ_j]=-[JU^\circ_i,JU^\circ_j]=\mathrm{Re}c_{ij}^k U^\circ_k+\mathrm{Im}c_{ij}^k JU^\circ_k, \\& &
[JU^\circ_i,U^\circ_j]=[U^\circ_i,JU^\circ_j]=-\mathrm{Im}c_{ij}^k U^\circ_k+
\mathrm{Re}c_{ij}^k JU^\circ_k,\\
& &[V^\circ_i,V^\circ_j]=-[JV^\circ_i,JV^\circ_j]=-\mathrm{Re}c_{ij}^k V^\circ_k-\mathrm{Im}c_{ij}^k JV^\circ_k,
\\& &
[JV^\circ_i,V^\circ_j]=[V^\circ_i,JV^\circ_j]= \mathrm{Im}c_{ij}^k V^\circ_k-
\mathrm{Re}c_{ij}^k JV^\circ_k.
\end{eqnarray*}

Since left invariant real tangent fields commute with right invariant real tangent fields, $U_i=\tfrac12(U_i^\circ-\sqrt{-1}JU_i^\circ)$ commutes with
$V_j=\tfrac12(V_j^\circ-\sqrt{-1}JV_j^\circ)$ by Lemma \ref{lemma-4}. More over, we have $\widetilde{U}_i$ commutes with $\widetilde{V}_j$ by Lemma \ref{lemma-11}.

By Lemma \ref{lemma-1} and Lemma \ref{lemma-3},
$$[U_i,U_j]^\circ=[U_i^\circ,U_j^\circ]=\mathrm{Re}c_{ij}^k U_k^\circ+\mathrm{Im}c_{ij}^k JU_k^\circ,
$$
which coincides with $(c_{ij}^k U_k)^\circ$, so we have $[U_i,U_j]=c_{ij}^k U_k$, and then Lemma \ref{lemma-11} provides
$[\widetilde{U}_i,\widetilde{U}_j]=c_{ij}^k\widetilde{U}_k$.

(3) Because
\begin{eqnarray*}
0=[U_i,V_j]=[U_i,\psi_j^kU_k]=\psi_j^k[U_i,U_k]+U_i(\psi_j^k)U_k=
(\psi_j^k c_{ik}^l+U_i(\psi_j^l))U_l,
\end{eqnarray*}
we get $U_i(\psi_j^k)=\psi_j^lc^k_{li}$, $\forall i,j,k$.
Since $\mathrm{Ad}(g^{-1}):\mathfrak{g}\rightarrow\mathfrak{g}$
is a Lie algebra isomorphism, i.e., $\psi_i^\alpha\psi_j^\beta c_{\alpha\beta}^\gamma=\psi_k^\gamma c_{ij}^k$, $U_i(\psi_j^k)$ can also be presented as $U_i(\psi_j^k)=\psi_j^lc^k_{li}=\phi_i^l\psi_p^k c_{jl}^p$.

(4)
By Lemma \ref{lemma-6}, the lift of $V_i=\psi_i^jU_j$ can be presented as
\begin{eqnarray*}
\widetilde{V}_i&=&\psi_i^j\widetilde{U}_j+u^j U_j(\psi_i^k)\partial_{u^k}
=\psi_i^j\widetilde{U}_j+u^j\psi_i^l c^k_{lj}\partial_{u^k}\\
&=&\psi_i^j\widetilde{U}_j+u^j\phi_j^l\psi_p^k c_{il}^p\partial_{u^k}
=\psi_i^j\widetilde{U}_j-c_{li}^k v^l \partial_{v^k}.
\end{eqnarray*}

(5) Since $\widetilde{V}_j$ generates (the tangent maps of) left translations on $T^{1,0}G$, which preserves all $U_i$, it also preserves all functions $u_i$, so we have $\widetilde{V}_j u^i=0$. Meanwhile, we have
$$\widetilde{V}_i v^j=\widetilde{V}_i(\phi^j_k u^k)=\widetilde{V}_i(\phi^j_k)u^k=\phi_k^l c_{il}^j u^k=-c_{li}^j v^l.
$$

(6) As in (5),  $[\widetilde{V}_i,\partial_{u^k}]=0$ can also be observed from the the left invariance of $u^k$. Notice that $\phi_i^k$ is only relevant to $z\in M$, so we have
$$[\widetilde{U}_i,\partial_{u^j}]=[\phi_i^k\widetilde{V}_k+
c^k_{li}u^l\partial_{u^k},\partial_{u^j}]=[
c^k_{li}u^l\partial_{u^k},\partial_{u^j}]=c^k_{ij}\partial_{u^k}.$$
The proof of Lemma \ref{lemma-8} is finished.
\end{proof}

Alternatively, we may use local coordinate $(z,w)=(z^1,\cdots,z^n,w^1,\cdots,w^n)$ to present these frames.
There exist $GL(n,\mathbb{C})$-valued holomorphic functions $(A^i_j)=(A^i_j(z))$,
$(B^i_j)=(B^i_j(z))=(A^i_j)^{-1}$, $(C^i_j)=(C^i_j(z))$ and $(D^i_j)=(D^i_j(z))=(C^i_j)^{-1}$, locally defined on $G$,
such that the following are satisfied,
\begin{eqnarray}
& &U_j=A^i_j\partial_{z^i}, \ \partial_{z^j}=B^i_j U_i,\
u^i=B^i_j w^j,\ w^i=A^i_j u^j, \ \partial_{u^j}=A^i_j\partial_{w^i}, \
\partial_{w^j}=B^i_j\partial_{u^i},\nonumber\\
& &V_j=C^i_j\partial_{z^i},\ \partial_{z^j}=D^i_j V_i,\
v^i=D^i_j w^j,\ w^i=C^i_j v^j, \ \partial_{v^j}=C^i_j\partial_{w^i}, \
\partial_{w^j}=D^i_j\partial_{v^i},\nonumber\\
& &A^i_jB^j_k=B^i_jA^j_k=
C^i_jD^j_k=D^i_jC^j_k=\delta^i_k, \
C_j^kB_k^i=\psi_j^i, \ A^k_jD_k^i=\phi_j^i.\label{016}
\end{eqnarray}
Lemma \ref{lemma-2} and Lemma \ref{lemma-6} provide
\begin{lemma}\label{lemma-12}
With respect to the local coordinate $(z,w)=(z^1,\cdots,z^n,w^1,\cdots,w^n)$ on $T^{1,0}G$, we have
\begin{eqnarray*}
\widetilde{U}_j=A^i_j\partial_{z^i}+w^i\partial_{z^i}A^k_j\partial_{w^k},& &
\partial_{z^j}=B^i_j\widetilde{U}_i+u^i U_i(B^k_j)\partial_{u^k},\\
\widetilde{V}_j=C^i_j\partial_{z^i}+w^i\partial_{z^i}C^k_j\partial_{w^k},
&\mbox{and}&
\partial_{z^j}=D^i_j\widetilde{V}_i+v^i V_i(D^k_j)\partial_{v^k},
\end{eqnarray*}
\end{lemma}
Let $F=\sqrt{\mathrm{G}}$
be a left invariant complex Finsler metric on $G$. We use
$(\mathcal{G}_{i\overline{j}})$ to denote the complex Hessian matrix
$(\partial_{u^i}\partial_{\overline{u}^j}\G)$ and
$(\mathcal{G}^{\overline{i}j})$ to denote the inverse matrix
$(\mathcal{G}_{i\overline{j}})$, i.e., $\mathcal{G}_{k\overline{j}}\mathcal{G}^{\overline{j}i}=
\mathcal{G}^{\overline{i}j}\mathcal{G}_{j\overline{k}}=\delta^i_k$.

\begin{lemma}\label{lemma-13}
With respect to the local coordinate $(z,w)=(z^1,\cdots,z^n,w^1,\cdots,w^n)$ on $T^{1,0}G$, we have
\begin{eqnarray}
& &\G_{i\overline{j}}=B_i^p\overline{B_j^q}
\mathcal{G}_{p\overline{q}},\quad
\G^{\overline{i}j}=\overline{A^i_p} A^j_q
\mathcal{G}^{\overline{p}q},\nonumber\\
& &u^i\partial_{u^i}\mathcal{G}_{j\overline{k}}
=u^j\partial_{u^i}\mathcal{G}_{j\overline{k}}
=\overline{u}^i\partial_{\overline{u}^i}\mathcal{G}_{j\overline{k}}
=\overline{u}^k\partial_{\overline{u}^i}\mathcal{G}_{j\overline{k}}=0,
\nonumber\\
& &\widetilde{V}_i\mathcal{G}_{j\overline{k}}\equiv0,
\quad\mbox{and}\quad
\widetilde{V}_i\mathcal{G}^{\overline{j}{k}}\equiv0.\nonumber
\end{eqnarray}
\end{lemma}
\begin{proof} The first line is obvious. The second follows immediately after the homogeneity property. The third is a corollary of Lemma \ref{lemma-5} and (6) of Lemma \ref{lemma-8}.
\end{proof}
\label{subsection-invariant-frame}
\section{Proof of the main theorems}

Let $G$ be a complex Lie group and $F$ a left invariant complex
Finsler metric on $G$. We will use the invariant frames and the notations
in the Section \ref{subsection-invariant-frame} to
calculate the connection coefficients and prove the main theorems.

\subsection{Berwald property and complex spray}
First, we prove Theorem \ref{main-thm-1}.

By Lemma \ref{lemma-8}, Lemma \ref{lemma-12} and Lemma \ref{lemma-13}, we have
\begin{eqnarray}
N^i_k&=&\G^{\overline{l}i}\partial_{\overline{w}^l}\partial_{z^k}\G
=\overline{A^l_p}A^i_q\mathcal{G}^{\overline{p}q}\overline{B_l^r}
\partial_{\overline{u}^r}((D_k^a\widetilde{V}_a+
v^aV_a(D_k^b)\partial_{v^b})\G)\nonumber\\
&=&\overline{A^l_p}A^i_q\mathcal{G}^{\overline{p}q}\overline{B_l^r}
\partial_{\overline{u}^r}((
v^aV_a(D_k^b)\psi_b^c\partial_{u^c})\G)
=A^i_q\mathcal{G}^{\overline{r}q}
v^aV_a(D_k^b)\psi_b^c\mathcal{G}_{c\overline{r}}\nonumber\\
&=&
A^i_cv^aV_a(D_k^b)\psi_b^c,\label{014}
\end{eqnarray}
Notice that the last equality of (\ref{016}) implies $A^i_j\psi^j_k=C^i_k$,
so we have
\begin{eqnarray}\label{017}
N^i_k=C^i_bv^aV_a(D_k^b)\quad\mbox{and}\quad \Gamma^i_{j;k}=\partial_{w^j}N^i_k=C^i_b D^a_j V_a(D_k^b).
\end{eqnarray}
From (\ref{017}), we see that $\Gamma^i_{j;k}$ only depends on $z\in M$.
So $F$ is complex Berwald. The proof of Theorem \ref{main-thm-1}
is finished.

Next, we calculate the complex spray
$\chi=w^k\delta_{z^k}=w^k(\partial_{z^k}-N^i_k\partial_{w^i})$ on $T^{1,0}G\backslash 0$ and prove Theorem \ref{main-thm-4}.

By Lemma \ref{lemma-12} and (\ref{014}), we have
\begin{eqnarray*}
w^k\delta_{z^k}&=&w^k(D_k^a\widetilde{V}_a+v^aV_a(D_k^c)\partial_{v^c})-
w^kA^i_cv^aV_a(D_k^b)\psi_b^c\partial_{w^i}\\
&=&w^aC^k_a(D_k^b\widetilde{V}_b)+ w^kv^bV_b(D_k^c)\partial_{v^c}
-w^kC^i_bv^aV_a(D_k^b)\partial_{w^i}\\
&=&v^a\widetilde{V}_a-w^kv^a(V_a(D_k^b)C_b^i-C^i_b V_a(D_k^b))\partial_{w^i}=v^i\widetilde{V}_i.
\end{eqnarray*}
So $\chi$ can be extended to a holomorphic tangent field $v^i\widetilde{V}_i$ on $T^{1,0}G$.

By Lemma \ref{lemma-7}, the real part of $\chi=v^i\widetilde{V}_i$ is $\tfrac{1}2(v^i_\Re\widetilde{V_i}^\circ+v^i_\Im J(\widetilde{V_i}^\circ))=
\tfrac{1}2(v^i_\Re\widetilde{V^\circ_i}+v^i_\Im \widetilde{J(V^\circ)}$, we get
$(v^i\widetilde{V}_i)^\circ=v^i_\Re\widetilde{V}^\circ_i+v^i_\Im J(\widetilde{V^\circ})$. When $G$ is viewed as a real Lie group, $v^i_\Re\widetilde{V}^\circ_i+v^i_\Im J(\widetilde{V^\circ})$ is the canonical bi-invariant spray structure in Theorem A of \cite{Xu2022-1}.
By similar calculation and \cite{Xu2022-1},
the canonical bi-invariant spray structure on $G$ can also be presented as $(u^i\widetilde{U}_i)^\circ
=u^i_\Re\widetilde{U^\circ_i}+u^i_\Im \widetilde{J(U_i^\circ)}$. So we have
$(u^i\widetilde{U})^\circ=(v^i\widetilde{V})^\circ$, i.e., $u^i\widetilde{U}=v^i\widetilde{V}$, which ends the proof of
Theorem \ref{main-thm-4}.

\subsection{Strongly K\"{a}hler, K\"{a}hler and weak K\"{a}hler properties}
Without loss of generality, we assume in this section that the Lie group $G$ is connected.

Notice that $F$ is K\"{a}hler if and only if it is strongly K\"{a}hler \cite{CS2009}, which can be equivalently described by the vanishing of
$(\Gamma^i_{j;k}-\Gamma^i_{k;j})\partial_{z^i}$ everywhere. By (\ref{017}), we have
\begin{eqnarray*}
(\Gamma^i_{j;k}-\Gamma^i_{k;j})\partial_{z^i}=
(C^i_b D^a_j V_a(D_k^b)-C^i_b D^a_k V_a(D_j^b))\partial_{z^i}
=(D^a_j V_a(D_k^b)-D^a_k V_a(D_j^b))V_b.
\end{eqnarray*}
On the other hand, the first line of (\ref{016}) provides
\begin{eqnarray*}
0&=&[\partial_{z^j},\partial_{z^k}]=[D^a_jV_a,D^b_kV_b]\\
&=&D^p_jD^q_k[V_p,V_q]+(D^a_jV_a(D^b_k)-D^a_kV_a(D^b_j))V_b\\
&=&-D^p_jD^q_k c_{pq}^rV_r+(D^a_jV_a(D^b_k)-D^a_kV_a(D^b_j))V_b.
\end{eqnarray*}
So we have
$(\Gamma^i_{j;k}-\Gamma^i_{k;j})\partial_{z^i}=D^p_jD^q_k c_{pq}^rV_r$, which vanishes if and only if all bracket coefficients $c^k_{ij}$ of $\mathfrak{g}$ vanishes, i.e., $\mathfrak{g}$ is Abelian. To summarize, we have proved the equivalence between $F$ is K\"{a}hler and $\mathfrak{g}$ is Abelian. To finish the proof of Theorem \ref{main-thm-3}, we only need to assume $F$ is weakly K\"{a}hler and prove $\mathfrak{g}$ is Abelian.

Together with (\ref{016}), above calculation implies that, at $e\in G$, we have
\begin{eqnarray}
w^k(\Gamma^i_{j;k}-\Gamma^i_{k;j})\G_i=
w^kD^p_jD^q_k c_{pq}^r\partial_{v^r}\G=D^p_j v^q c_{pq}^r \partial_{v^r}\G=0.
\label{018}
\end{eqnarray}
If $F$ is viewed as a complex Minkowski norm on the complex $\mathfrak{g}$, then (\ref{018}) is equivalent
to the vanishing of the directional derivative
$[e_p,v]\G=v^q c_{pq}^r\partial_{v^r}\G$, for any $p$, at each point $v=v^ie_i\in\mathfrak{g}\backslash\{0\}$.
Notice that $\mathfrak{g}$ may be viewed as a real Lie algebra, i.e., $\mathfrak{g}_{\mathbb{R}}=T_eG$, and the correspondence $u^\circ\in T_eG=\mathfrak{g}_\mathbb{R}\mapsto u=\tfrac12(u^\circ-\sqrt{-1}J(u^\circ))\in T^{1,0}_eG=\mathfrak{g}$ is a real Lie algebra
isomorphism. So when $\G$ is viewed as a function on $\mathfrak{g}_\mathbb{R}$, (\ref{018}) implies $[e_p^\circ,v^\circ]\G=[J(e_p^\circ),v^\circ]\G=0$ for every $p$, at each $v^\circ\in\mathfrak{g}_\mathbb{R}\backslash \{0\}$. By this observation, we have verified that $F$ is an $\mathrm{Ad}(G)$-invariant function on $\mathfrak{g}_\mathbb{R}$.

Now we construct an $\mathrm{Ad}(G)$-invariant Euclidean norm on $\mathfrak{g}_\mathbb{R}$ as follows. Since the indicatrix $F=1$ in $\mathfrak{g}_{\mathbb{R}}$ is a sphere surrounding the origin, its convex hull $\mathcal{C}$ contains the origin as an internal point. Then there exists
a real Minkowski norm $F'$ such that the indicatrix $F'=1$ coincides with the boundary sphere $\partial \mathcal{C}$.
Using the Binet-Legendre transformation, we get an Euclidean norm $F''$ on $\mathfrak{g}_\mathbb{R}$. In each step of above construction, the $\mathrm{Ad}(G)$-invariance can be preserved. To summarize, the existence of an $\mathrm{Ad}(G)$-invariant Euclidean norm implies that $\mathfrak{g}_\mathbb{R}$ is a compact Lie algebra. But the complex Lie algebra $\mathfrak{g}$ can only be compact when it is Abelian.
The proof of Theorem \ref{main-thm-3} is finished.
\subsection{The vanishing of holomorphic sectional curvature}
To prove Theorem \ref{main-thm-3}, we need to use (\ref{016}), Lemma \ref{lemma-12} and Lemma \ref{lemma-13} to translate
\begin{eqnarray}\label{021}
w^i\overline{w}^j {w}^k\overline{w}^l
R_{i\overline{j}k\overline{l}}
=-w^i\overline{w}^j{w}^k\overline{w}^l
\G_{i\overline{j};k\overline{l}}
+
w^i\overline{w}^j {w}^k\overline{w}^l
\G^{\overline{q}{p}}\G_{i\overline{q};k}
\G_{p\overline{j};\overline{l}}.
\end{eqnarray}

First, we have
\begin{eqnarray}
& &w^i\overline{w}^j{w}^k\overline{w}^l
\G_{i\overline{j};k\overline{l}}=w^i\overline{w}^j{w}^k\overline{w}^l\partial_{z^k}\partial_{\overline{z}^l}
(B_i^p\overline{B_j^q}\mathcal{G}_{p\overline{q}})\nonumber\\
&=&w^i\overline{w}^j{w}^k\overline{w}^l\partial_{z^k}B_i^p\partial_{\overline{z}^l}\overline{B^q_j}
\mathcal{G}_{p\overline{q}}
+w^i\overline{w}^j{w}^k\overline{w}^lB^p_i\overline{B^q_j}\partial_{z^k}\partial_{\overline{z}^l}
\mathcal{G}_{p\overline{q}}\nonumber\\
& &+w^i\overline{w}^j{w}^k\overline{w}^l\partial_{z^k}B^p_i \overline{B^q_j}\partial_{\overline{z}^l}\mathcal{G}_{p\overline{q}}
+w^i\overline{w}^j{w}^k\overline{w}^lB^p_i\partial_{\overline{z}^l}\overline{B^q_j}\partial_{z^k}
\mathcal{G}_{p\overline{q}}.\label{020}
\end{eqnarray}

By Theorem \ref{main-thm-2} and
Lemma \ref{lemma-12}, we have
\begin{eqnarray}\label{022}
w^k\partial_{z^k}&=&w^k B_k^a\widetilde{U}_a+w^ku^aU_a(B^b_k)\partial_{u^b}\nonumber\\
&=&u^a\widetilde{U}_a
+w^k u^aU_a(B^b_k)\partial_{u^b}=v^a\widetilde{V}_a+w^k u^aU_a(B^b_k)\partial_{u^b},
\end{eqnarray}
so the second summand in the right of (\ref{020}) satisfies
\begin{eqnarray*}
& &w^i\overline{w}^j{w}^k\overline{w}^lB^p_i\overline{B^q_j}\partial_{z^k}
\partial_{\overline{z}^l}
\mathcal{G}_{p\overline{q}}
=\overline{u}^q\overline{w}^l\partial_{\overline{z}^l}(u^pw^k\partial_{z^k}
\mathcal{G}_{p\overline{q}})\\
&=&\overline{u}^q\overline{w}^l\partial_{\overline{z}^l}
(u^pv^a\widetilde{V}_a\mathcal{G}_{p\overline{q}}+w^k u^aU_a(B^b_k)u^p\partial_{u^b}\mathcal{G}_{p\overline{q}})=0,
\end{eqnarray*}
where we have applied Lemma \ref{lemma-13} for the last two equalities.

Next, using Lemma \ref{lemma-13}, we can translate the second summand in the right side of (\ref{021}) as
\begin{eqnarray}
& &w^i\overline{w}^jw^k\overline{w}^l
\G^{\overline{q}p}\G_{i\overline{q};k}\G_{p\overline{j};\overline{l}}\nonumber\\
&=&w^i\overline{w}^jw^k\overline{w}^l\overline{A^q_a}A^p_b\mathcal{G}^{\overline{a}b}
\partial_{z^k}(B_i^s\overline{B^t_q}\mathcal{G}_{s\overline{t}})
\partial_{\overline{z}^l}(B_p^\alpha\overline{B_j^\beta}
\mathcal{G}_{\alpha\overline{\beta}})\nonumber\\
&=&w^i\overline{w}^jw^k\overline{w}^l\mathcal{G}^{\overline{a}b}
\partial_{z^k}(B_i^s\mathcal{G}_{s\overline{a}})
\partial_{\overline{z}^l}(\overline{B_j^\beta}
\mathcal{G}_{b\overline{\beta}})\nonumber\\
&=&w^i\overline{w}^jw^k\overline{w}^l
\partial_{z^k}B^s_i\partial_{\overline{z}^l}\overline{B^\beta_j}
\mathcal{G}_{s\overline{\beta}}
+w^i\overline{w}^jw^k\overline{w}^l
\partial_{z^k}B^b_i \overline{B^\beta_j}\partial_{\overline{z}^l}
\mathcal{G}_{b\overline{\beta}}\nonumber\\
& &+w^i\overline{w}^jw^k\overline{w}^l
B^s_i\partial_{\overline{z}^l}\overline{B^a_j}
\partial_{z^k}\mathcal{G}_{s\overline{a}}
+w^i\overline{w}^jw^k\overline{w}^l
B^s_i\overline{B^\beta_j}\mathcal{G}^{\overline{a}b}
\partial_{z^k}\mathcal{G}_{s\overline{a}}\partial_{\overline{z}^l}
\mathcal{G}_{b\overline{\beta}}.\label{023}
\end{eqnarray}

The first three summands in the right side of (\ref{023}) cancel
the first, the third and the fourth summands in the side of (\ref{020}).
So we have
$
w^i\overline{w}^j {w}^k\overline{w}^l
R_{i\overline{j}k\overline{l}}
=w^i\overline{w}^jw^k\overline{w}^l
B^c_i\overline{B^d_j}\mathcal{G}^{\overline{a}b}
\partial_{z^k}\mathcal{G}_{c\overline{a}}\partial_{\overline{z}^l}
\mathcal{G}_{b\overline{d}}$. Using (\ref{016}), Lemma \ref{lemma-13} and (\ref{022}) again, we get
\begin{eqnarray*}& &
w^i\overline{w}^j {w}^k\overline{w}^l
R_{i\overline{j}k\overline{l}}
=w^i\overline{w}^jw^k\overline{w}^l
B^c_i\overline{B^d_j}\mathcal{G}^{\overline{a}b}
\partial_{z^k}\mathcal{G}_{c\overline{a}}\partial_{\overline{z}^l}
\mathcal{G}_{b\overline{d}}\\&
=&u^c\overline{u}^d \mathcal{G}^{\overline{a}b}
(v^r\widetilde{V}_r(\mathcal{G}_{c\overline{a}})
+w^ru^sU_s(B_r^t)\partial_{u^t}\mathcal{G}_{c\overline{a}})
\overline{({v}^\alpha\widetilde{V}_\alpha(\mathcal{G}_{d\overline{b}})+
w^\alpha u^\beta U_\beta(B^\gamma_\alpha)\partial_{u^\gamma}\mathcal{G}_{d\overline{b}})}\\
&=&\mathcal{G}^{\overline{a}b}
(w^ru^sU_s(B_r^t)u^c\partial_{u^t}\mathcal{G}_{c\overline{a}})
\overline{(
w^\alpha u^\beta U_\beta(B^\gamma_\alpha)u^d\partial_{u^\gamma}\mathcal{G}_{d\overline{b}})}
=0.
\end{eqnarray*}
The proof of Theorem \ref{main-thm-4} is finished.
\section{Appendix: complete lift of a real tangent field}
Let $M$ be a smooth real manifold with $\dim M=n$ and $V$ a smooth real tangent field on $M$.
The local coordinate presentation for the lift $\widetilde{V}$ of $V$ to $M$ appears in \cite{XD2014} and \cite{Xu2022-1}. To be self contained, we present a proof here.


\begin{lemma}\label{lemma-9}
For $V=f^i\partial_{x^i}$, we have $\widetilde{V}=f^i\partial_{x^i}+y^i\partial_{x^i}f^j\partial_{y^j}$.
\end{lemma}

\begin{proof}
First we briefly recall how $\widetilde{V}$ is defined. The tangent field $V$ generates a flow of local diffeomorphisms $\rho_t$ on $M$. Next, the tangent maps $(\rho_t)_*$ provide a flow of local diffeomorphisms on $TM$. At last, $\widetilde{V}=\frac{{\rm d}}{{\rm d}t}(\rho_t)_*$.

Now we impose the local coordinate
$x=(x^1,\cdots,x^n)\in$ where $V$ is defined. When $t$ sufficiently close to $0$, we denote by $c(x,t)=(c^1(x,t),\cdots,c^n(x,t))$ the integral curve of $V$, which is the solution of
\begin{eqnarray}\label{009}
\partial_{t} c^i(x,t)=f^i(c(x,t)),\ \forall i,\quad \mbox{and}\quad c(x,0)=x.
\end{eqnarray}
Replace $x$ in (\ref{009}) with a smooth curve $x(s)$ which represents the tangent vector $\frac{{\rm d}}{{\rm d}s}x(0)=y=(y^i,\cdots,y^n)$ at $x=x(0)$,
and then apply $\partial_s|_{s=0}$ to (\ref{009}),  we get
\begin{eqnarray}\label{010}
\frac{\rm d}{{\rm d}t} c'{}^i(t)= c'{}^j(t)\partial_{c^j}f^i(c(x,t)),\ \forall i,\quad\mbox{and}\quad c'(0)=y,
\end{eqnarray}
with ${c}'(t)=\partial_s|_{s=0}c(x(s),t)$.
In particular, $\frac{{\rm d}}{{\rm dt}}|_{t=0}c'{}^i(t)=y^j\partial_{x^j}f^i(x)$. So
\begin{eqnarray*}
\widetilde{V}(x,y)&=&\frac{{\rm d}}{{\rm d}t}|_{t=0}(\rho_t)_*(x,y)=(\frac{{\rm d}}{{\rm d}t}|_{t=0}c(t),\frac{{\rm d}}{{\rm d}t}|_{t=0}c'(t))\\&=&
(y^1,\cdots,y^n,y^j\partial_{x^j}f^1,\cdots,y^j\partial_{x^j}f^n)
\end{eqnarray*}
has the local presentation $\widetilde{V}=y^i\partial_{x^i}+y^j\partial_{x^j}V^i\partial_{y^i}$
\end{proof}

Fundamental knowledge in Lie theory implies that lifts commute with the bracket. Here we present an explicit proof in the context.

\begin{lemma}\label{lemma-10}
For smooth real tangent fields $U$ and $V$ on $M$, we have
$[\widetilde{U},\widetilde{V}]=\widetilde{[U,V]}$.
\end{lemma}
\begin{proof}Assume the following local presentations, $U=f^i\partial_{x^i}$ and $V=g^i\partial_{x^i}$. Then the lift of $[U,V]=(f^i\partial_{x^i}g^j-g^i\partial_{x^i}f^j)\partial_{x^j}$
is given by
\begin{eqnarray}
\widetilde{[U,V]}&=&[(f^i\partial_{x^i}g^j-g^i\partial_{x^i}f^j)\partial_{x^j}
+y^k\partial_{x^k}(f^i\partial_{x^i}g^j-g^i\partial_{x^i}f^j)\partial_{y^j}]\nonumber
\\
&=&(f^i\partial_{x^i}g^j-g^i\partial_{x^i}f^j)\partial_{x^j}
+y^k(\partial_{x^k}f^i\partial_{x^i}g^j-\partial_{x^k}g^i\partial_{x^i}f^j)
\nonumber\\
& &+y^k(f^i\partial_{x^k}\partial_{x^i}g^j-
g^i\partial_{x^k}\partial_{x^i}f^j)\partial_{y^j}.\label{011}
\end{eqnarray}
On the other hand,
\begin{eqnarray}
[\widetilde{U},\widetilde{V}]
&=&[f^i\partial_{x^i}+y^k\partial_{x^k}f^j\partial_{y^j},
g^i\partial_{x^i}+y^k\partial_{x^k}g^j\partial_{y^j}]\nonumber\\
&=&(f^i\partial_{x^i}g^j-g^i\partial_{x^i}f^j)\partial_{x^j}+
y^k(f^i\partial_{x^i}\partial_{x^k}g^j-g^i\partial_{x^i}\partial_{x^k}f^j)
\nonumber\\
& &+y^k(\partial_{x^k}f^i\partial_{x^i}g^j-\partial_{x^k}g^i\partial_{x^i}f^j)
\partial_{y^j}.\label{012}
\end{eqnarray}
Comparing (\ref{011}) and (\ref{012}), Lemma \ref{lemma-10} is proved.
\end{proof}
\noindent

{\bf Acknowledgement}\quad
This paper is supported by National Natural Science Foundation of China (No. 12131012). The authors sincerely thank Liyou Zhang and Chunping Zhong for helpful discussions.

\end{document}